\documentclass[12pt]{amsart}
\usepackage{amscd}
\usepackage{amsmath}
\usepackage{verbatim}
\usepackage{amssymb}
\usepackage{amsxtra}

\usepackage{amssymb}

\DeclareMathAlphabet{\cat}{OT1}{cmss}{m}{sl}

\numberwithin{equation}{section}

\newtheorem{thm}[equation]{Theorem}

\newtheorem{lemma}[equation]{Lemma}
\newtheorem{cor}[equation]{Corollary}

\theoremstyle{definition}
\newtheorem{rem}[equation]{Remark}
\newtheorem{example}[equation]{Example}
\newtheorem{dfn}[equation]{Definition}

\newtheorem{problem}[equation]{Problem}

\newcommand{\onto}{\rightarrow\!\!\rightarrow}
\newcommand{\codim}{\operatorname{codim}}

\newcommand{\colim}{\operatorname{colim}}

\newcommand{\SB}{X}

\newcommand{\ind}{\mathop{\mathrm{ind}}}

\newcommand{\rk}{\mathop{\mathrm{rk}}}
\newcommand{\CH}{\mathop{\mathrm{CH}}\nolimits}

\newcommand{\PGL}{\operatorname{\mathrm{PGL}}}

\newcommand{\Ch}{\mathop{\mathrm{Ch}}\nolimits}
\newcommand{\BCh}{\mathop{\overline{\mathrm{Ch}}}\nolimits}

\newcommand{\res}{\mathop{\mathrm{res}}\nolimits}

\newcommand{\cd}{\mathop{\mathrm{cdim}}\nolimits}

\newcommand{\mult}{\operatorname{mult}}

\newcommand{\Z}{\mathbb{Z}}
\newcommand{\F}{\mathbb{F}}

\newcommand{\Spec}{\operatorname{Spec}}
\newcommand{\End}{\operatorname{End}}
\newcommand{\Hom}{\operatorname{Hom}}
\newcommand{\Aut}{\operatorname{Aut}}

\newcommand{\compose}{\circ}

\newcommand{\CM}{\operatorname{CM}}

\newcommand{\corr}{\rightsquigarrow}

\renewcommand{\phi}{\varphi}

\newcommand{\RatM}{\dashrightarrow}

\newcommand{\cT}{T}

\newcommand{\cX}{\mathfrak{X}}

\newcommand{\BCM}{\operatorname{\overline{CM}}}

\usepackage[hypertex]{hyperref}

\title
[Upper motives and
incompressibility]
{Upper motives of algebraic groups and \\
incompressibility of Severi-Brauer varieties}


\author
{Nikita A. Karpenko}

\address
{
UPMC Sorbonne Universit\'es\\
Institut de Math\'ematiques de Jussieu\\
F-75252 Paris\\
FRANCE}

\address
{{\it Web page:}
{\tt www.math.jussieu.fr/\~{ }karpenko}}

\email {karpenko {\it at} math.jussieu.fr}

\date{18 April 2009. Updated: 9 July 2011}

\thanks
{Supported by the Collaborative Research Centre 701 of the Bielefeld University}

\begin{document}

\begin{abstract}
Let $G$ be a semisimple affine algebraic group of inner type over a field $F$.
We write $\cX_G$ for the class of all finite direct products of projective $G$-homogeneous
$F$-varieties.
We determine the structure of the {\em Chow motives with coefficients in a finite field} of the varieties in $\cX_G$.
More precisely, it is known that
the motive of any variety in $\cX_G$ decomposes (in a unique way) into a sum of
indecomposable motives, and
we describe the indecomposable summands
which appear in the decompositions.

In the case where $G$ is the group $\PGL A$ of automorphisms of
a given central simple $F$-algebra $A$,
for any variety in the class $\cX_G$ (which includes the {\em generalized Severi-Brauer varieties}
of the algebra $A$) we determine its {\em canonical dimension at any prime $p$}.
In particular, we find out which varieties in $\cX_G$ are {\em $p$-incompressible}.
If $A$ is a division algebra of degree $p^n$ for some $n\geq0$, then the list of
$p$-incompressible varieties includes the generalized Severi-Brauer variety $X(p^m;A)$ of
ideals of reduced dimension $p^m$ for
$m=0,1,\dots,n$.
\end{abstract}

\maketitle


\section
{Introduction}
\label{Introduction}

A smooth complete irreducible variety $X$ over a field $F$ is said to be {\em incompressible},
if any rational map $X\RatM X$ is dominant.

An important (at least by the amount of available applications, see, e.g.,
\cite{MR1751923}, \cite{MR2358058}, or \cite{BRV})
example of an incompressible variety is as follows.
Let $p$ be a positive prime integer.
Let $D$ be a central division $F$-algebra of degree a power of $p$, say $p^n$.
For any integer $i$, the {\em generalized Severi-Brauer variety} $X(i;D)$ is the $F$-variety
of right ideals in $D$ of reduced dimension $i$.
The variety $X(1;D)$
(the usual Severi-Brauer variety of $D$) is incompressible.

In fact, the variety $X=X(1;D)$ has a stronger property -- {\em $p$-incompressibility},
defined in \S\ref{Canonical dimension}.
The original proof, given in \cite{MR1356536} and \cite{MR1751923}, makes use of Quillen's computation
of $K$-theory of Severi-Brauer varieties.
A recent proof, given in \cite{MR2024643},
makes use of Steenrod operations on the modulo $p$ Chow groups
(it explains the reason of the $p$-incompressibility better but works only over fields of characteristic $\ne p$).
A third particularly simple proof is given here (see Corollaries \ref{SB-cor} and \ref{indecomposable-incompressible}).

One of the main results of the present paper is the incompressibility theorem (Theorem \ref{main})
which affirms that the variety $X(p^m;D)$ is
$p$-incompressible also for $m=1,\dots,n-1$
(in the case of $p=2$ and $m=n-1$ this was shown earlier by Bryant Mathews
\cite{MR2713320}
using a different method).
The remaining results of this paper are either obtained on the way to the
main result or are
consequences of it.

We start in the general context of an arbitrary semisimple affine algebraic group $G$
of inner type over a field $F$.
Let $T$ be the set of the conjugacy classes of the maximal parabolic subgroups in $\bar{G}=G_{\bar{F}}$
($T$ can be identified with the set of vertices of the Dynkin diagram of $G$).
To each  $\tau\subset T$, a projective $G$-homogeneous $F$-variety $X_\tau$ is associated in a standard
way.
We define an indecomposable motive $M_\tau$
in the category of Chow motives with coefficients in $\F_p$, as the summand
in the complete motivic decomposition of $X_\tau$
such that the $0$-codimensional Chow group of $M_\tau$ is non-zero.
We show (see Theorem \ref{Gbasic1})
that the motive of any finite direct product of projective $G$-homogeneous $F$-varieties
decomposes into a sum of shifts of the motives $M_\tau$ (with various $\tau$).
Therefore we solve the inner case of the following problem:

\begin{problem}
Let $\cX_G$ be the class of all finite direct products of projective homogeneous varieties
under an action of a semisimple affine algebraic group $G$.
According to \cite{MR2264459} (see also \S\ref{Chow motives with finite coefficients} here),
the motive (still with $\F_p$ coefficients,
$p$ a fixed prime) of any variety in $\cX_G$
decomposes and in a unique way in a finite direct sum of indecomposable motives.
The problem is to describe the indecomposable summands which appear this way.
\end{problem}

Our method can also be applied to the groups of outer type.
Since we are mainly interested in the inner type $\cat{A}$ and for the sake of simplicity,
we do not work with groups of outer type in this paper.\footnote{The outer case has been treated in
\cite{outer}, a successor paper.}

For $G=\Aut D$, the above result translates as follows (see Theorem \ref{basic1}, where a $p$-primary
division algebra $D$ is replaced by a more general object, an arbitrary central simple algebra $A$).
To each integer $m=0,1,\dots,n$, an indecomposable motive $M_{m,D}$
in the category of Chow motives with coefficients in $\F_p$ is associated.
This is the summand in the complete motivic decomposition of the variety $X(p^m;D)$
such that the $0$-codimensional Chow group of $M_{m,D}$ is non-zero.
The motive of any variety in $\cX_D:=\cX_G$
decomposes into a sum of shifts of the motives $M_{m,D}$
(with various $m$).

With this in hand, we prove two structural results concerning the motives $M_{m,D}$
(Theorem \ref{basic2}).
We show that
the $d$-dimensional Chow group of $M_{m,D}$, where $d=\dim X(p^m;D)$, is also non-zero.
This result is equivalent to the $p$-incompressibility of the variety $X(p^m;D)$,
so that we get the incompressibility theorem (Theorem \ref{main}) at this point.
The second structural result on the motive $M_{m,D}$ is a computation of the $p$-adic
valuation of its rank.
In fact, we cannot separate the proofs of these two structural results.
We prove them
simultaneously by induction on $\deg D$ (and using Theorem \ref{basic1}).

An immediate consequence of the incompressibility theorem is as follows.
We recall (in \S\ref{Canonical dimension})
the notion of {\em canonical dimension at $p$} (or {\em canonical $p$-dimension})
$\cd_p(X)$ of a smooth complete irreducible
algebraic variety $X$.
This is a certain non-negative integer satisfying $\cd_p X\leq \dim X$;
moreover, $\cd_p X=\dim X$ if and only if
$X$ is $p$-incompressible.
In particular, by our main result,
$
\cd_p X(p^m;D)=\dim X(p^m;D)=p^m(p^n-p^m).
$
The canonical dimension at $p$ of any
variety in $\cX_A:=\cX_{\Aut A}$, where $A$ is an arbitrary central simple $F$-algebra,
can easily be computed in terms of $\cd_p X(p^m;D)$, where $D$ is the $p$-primary part of a division algebra
Brauer-equivalent to $A$ (see Corollary \ref{cor basic2}).

In spite of a large number of results obtained,
one may say that (the motivic part of) this paper raises more questions than it answers.
Indeed, although we show that the motives of the varieties in $\cX_G$ decompose into
sums of shifts of $M_\tau$ (and find  a restriction on $\tau$ in terms of a given variety), we
do not precisely determine this decomposition (even for $G$ simple of inner type $\cat{A}$):
we neither know how many copies of
$M_\tau$ (for a given $\tau$ and a given variety) do really appear in the decomposition, nor do we
determine the shifting numbers.
Moreover, the understanding of the structure of the motives $M_\tau$ themselves, which we provide for
$G$ simple of inner type $\cat{A}$  (our main case of interest), is not satisfactory.
It could be that $M_{m,D}$ is always the whole motive of the variety $X(p^m,D)$ (that
is, the motive of this variety probably is indecomposable): we do not possess a single
counter-example.
In fact, the variety $X(p^m,D)$ {\em is} indecomposable for certain values of $p$,
$n$, and $m$.
Two cases have been known for a long time: $m=0$ (the Severi-Brauer case, see Corollary \ref{SB-cor})
and $m=1$ with $p=2=n$ (reducing the exponent of $D$ to $2$, we come to the case of an Albert
quadric here).
The Albert case is generalized to arbitrary $n$ in Theorem \ref{2-2^n}.
The remaining values of $p,n,m$ should be studied in this regard.\footnote{Recently, Maksim Zhykhovich has shown
that for each of the remaining values of $p,n,m$ the corresponding variety {\em is decomposable}, \cite{maksim}.}

But the qualitative
analysis is done
(for instance, the properties of $M_{m,D}$, we establish, show that this motive behaves
essentially like the whole motive of the variety $X(p^m,D)$ even if it is ``smaller'').
And the proofs are not complicated.
They consist in a study of generalized  Severi-Brauer varieties which are twisted forms of
grassmannians, and there is no single Young diagram in the text!
Combinatorics or complicated formulas
do not show up at all, in particular, because we (can) neglect the shifting
numbers of motivic summands in most places.
The results we are getting this way are less precise but, as we believe, they contain
the essential piece of information.
They can be (and are) applied (in \cite{hypernew-tignol}) to prove the hyperbolicity
conjecture on orthogonal involutions.\footnote{They are also applied in \cite{oddisotro-tignol}
to prove the isotropy conjecture on involutions and in \cite{qweil} to prove incompressibility of
certain Weil transfers of generalized Severi-Brauer varieties.
For some further applications see \cite{gog}, \cite{og}, \cite{gug}, \cite{unitary}, \cite{isouni}.}
(This result has been recently expanded by Jean-Pierre Tignol to symplectic and unitary involutions.)

We conclude the introduction by some remarks on the motivic category we are using.
First of all, the category of Chow motives with coefficients in $\F_p$
(or, slightly more general, with coefficients in a finite connected commutative ring $\Lambda$),
in which we are working in this paper, can be replaced by a simpler category.
This simpler category is constructed in exactly the same way as the category of
Chow motives with the only difference that one kills the elements of Chow groups
which vanish over some extension of the base field (see Remark \ref{reduced motives}).
Working with this simpler category, we do not need the nilpotence tricks
(the nilpotence theorem and its standard consequences,
cf.\!\! \S\ref{Chow motives with finite coefficients}) anymore.
This simplification of the motivic category isn't harmful to any external application of our motivic results.
So, it is more for a question of taste than for a question of necessity that we stay with the usual Chow motives.

On the other hand, somebody may think that our category of usual Chow motives is not
honest or usual enough because these are Chow motives with coefficients in $\F_p$ and not in
$\Z$.
Well, there are at least three arguments here.
First, decompositions into sums of indecomposables are not unique for
coefficients in $\Z$, even in the case of projective homogenous varieties of inner type $\cat{A}$
(see \cite[Example 32]{MR2264459} or
\cite[Corollary 2.7]{MR2249542}).
Therefore the question of describing the indecomposables does not seem
so reasonable for the integral motives.
Second, any decomposition with coefficients in $\F_p$ lifts (and in a
unique way) to the coefficients $\Z/p^n\Z$ for any $n\geq2$,
\cite[Corollary 2.7]{J-invariant}.
Moreover, it also lifts to $\Z$ (non-uniquely this time) in the case of varieties in $\cX_G$, where $G$ is
a semisimple affine algebraic group of inner type for which $p$ is the unique torsion prime,
\cite[Theorem 2.16]{J-invariant}.
And third, maybe the most important argument, is that the results on motives with coefficients in $\F_p$
are sufficient for the applications.

\bigskip
\noindent
{\sc Acknowledgments.}
Fruitful and stimulating discussions with Alexander Vishik and Kirill Zainoulline  helped to improve the initial version of this manuscript.
A particular gratitude goes to Kirill Zainoulline for noticing
that the statement and the proof of Theorem \ref{Gbasic1}, given initially only for
a simple algebraic groups of inner type $\cat{A}$, are valid for an arbitrary semisimple group.
Thanks to Mark MacDonald for finding numerous typos and a mistake in an earlier version of the paper.

\section
{Preliminaries}

This section is long because it also includes some non-standard (but simple) material.

\subsection
{Chow motives with finite coefficients}
\label{Chow motives with finite coefficients}

Our basic reference for Chow groups and Chow motives (including notation)
is \cite{EKM}.
We fix an
associative unital commutative ring $\Lambda$,
and for a variety (i.e., a separated scheme of finite type over a field)
$X$ we write $\Ch(X)$ for its Chow group with coefficients in
$\Lambda$ (while we write $\CH(X)$ for its integral Chow group).
Our category of motives is
the category $\CM(F,\Lambda)$ of {\em graded Chow motives with
coefficients in $\Lambda$,} \cite[definition of \S64]{EKM}.
By a {\em sum} of motives we always mean a {\em direct} sum.
We also write $\Lambda$ for the motive $M(\Spec F)\in\CM(F,\Lambda)$.
A {\em Tate motive} is a motive $\Lambda(i)$ with $i$ an integer (which may differ from $\pm1$).

We shall often assume that our coefficient ring $\Lambda$ is finite.
This simplifies significantly the situation (and is sufficient for most applications).
For instance, for a finite $\Lambda$, the endomorphism rings of finite sums
of
Tate motives
are also finite and the following easy statement applies:

\begin{lemma}
\label{finite}
An appropriate power of any element of any {\em finite} associative (not necessarily
commutative) ring is idempotent.
\end{lemma}

\begin{proof}
Since the ring is finite, any element $x$ satisfies $x^a=x^{a+b}$
for some $a\geq1$ and $b\geq1$.
It follows that $x^{ab}$ is idempotent.
\end{proof}

Let $X$ be a smooth complete variety over $F$ and let $M$ be a motive.
We call $M$ {\em split}, if it is a finite sum of Tate motives.
We call $X$ {\em split}, if its {\em integral} motive $M(X)\in\CM(F,\Z)$
(and therefore the motive of $X$ with an arbitrary coefficient ring $\Lambda$)
is split.
We call $M$ or $X$ {\em geometrically split}, if it splits over
a field extension of $F$.
We say that $X$ satisfies the {\em nilpotence principle}, if for any field extension
$E/F$ and any coefficient ring $\Lambda$,
the kernel of the change of field homomorphism $\End(M(X))\to\End(M(X)_E)$ consists
of nilpotents.
Any projective homogeneous (under an action of a semisimple affine algebraic group)
variety is geometrically split and satisfies the nilpotence principle,
\cite{MR2110630}.

\begin{cor}
\label{cor-finite}
Assume that the coefficient ring $\Lambda$ is finite.
Let $X$ be a geometrically split variety satisfying the nilpotence principle.
Then an appropriate power of any endomorphism of the motive of $X$
is a projector.
\end{cor}

\begin{proof}
Let $\bar{F}/F$ be a splitting field of the motive $M(X)$, that is, $M(X)_{\bar{F}}$ is a
sum of Tate motives.
Let $f$ be an endomorphism of $M(X)$.
Since $\Lambda$ is finite, the ring
$\End(M(X)_{\bar{F}})$ is finite.
Therefore a power of $f_{\bar{F}}$ is idempotent by Lemma \ref{finite},
and (replacing $f$ by an appropriate power of $f$) we may assume that $f_{\bar{F}}$ is idempotent.
Since $X$ satisfies the nilpotence principle,
the element $\varepsilon:=f^2-f$ is nilpotent.
Let $n$ be a positive integer such that $\varepsilon^n=0=n\varepsilon$.
Then $(f+\varepsilon)^{n^n}=f^{n^n}$ because the binomial coefficients
$\binom{n^n}{i}$ for $i<n$  are divisible by $n$.
Therefore $f^{n^n}$ is a projector.
\end{proof}

\begin{lemma}[{cf.\! \cite[Theorem 28]{MR2264459}}]
\label{exists}
Assume that the coefficient ring $\Lambda$ is finite.
Let $X$ be a geometrically split variety satisfying the nilpotence principle and let
$\pi\in\End(M(X))$ be a projector.
Then the motive $(X,\pi)$ decomposes into a finite sum of indecomposable motives.
\end{lemma}

\begin{proof}
If $(X,\pi)$ does not decompose this way, we get an infinite sequence
$$\pi_0\!=\!\pi,\; \pi_1,\;\pi_2,\;\dots\;\;\in\End(M(X))$$
of pairwise distinct projectors such that
$\pi_i\compose \pi_j=\pi_j=\pi_j\compose \pi_i$ for any $i<j$.

Let $\bar{F}/F$ be a splitting field of $X$.
Since the ring $\End(M(X)_{\bar{F}})$ is finite, we have $(\pi_i)_{\bar{F}}=(\pi_j)_{\bar{F}}$ for
some $i<j$.
The difference $\pi_i-\pi_j$ is nilpotent and idempotent, therefore $\pi_i=\pi_j$.
\end{proof}

A (not necessarily commutative) ring is called {\em local}, if the sum of any
two non-invertible elements differs from $1$ in the ring.
Since the sum of two nilpotents is never $1$, we have

\begin{lemma}
\label{asm}
A ring, where each non-invertible element is nilpotent, is local.
In particular, by Corollary \ref{cor-finite}, so is the ring
$\End(M(X))$, if $\Lambda$ is finite and $X$ is a geometrically split variety satisfying the nilpotence principle and such that the motive $M(X)$ is indecomposable.
\qed
\end{lemma}

A {\em complete decomposition} of an object in an additive category is a finite
direct sum decomposition with indecomposable summands.

\begin{thm}[{\cite[Theorem 3.6 of Chapter I]{MR0249491}}]
\label{bass}
Let $M$ be an object of a pseudo-abelian category which is a direct sum of a finite number of indecomposable
objects having {\em local} endomorphism rings.
Then any finite direct sum decomposition of $M$ can be refined to a complete one, and
there is only one (up to a permutation of the summands) complete decomposition of $M$.
\end{thm}

To be precise, the uniqueness part of Theorem \ref{bass} states that if
$$
M=M_1\oplus\dots\oplus M_m=N_1\oplus\dots\oplus N_n
$$
are two complete decompositions of $M$, then $m=n$ and there exists a permutation
$\sigma$ of the set $\{1,2,\dots,n\}$ such that $M_i\simeq N_{\sigma(i)}$ for any $i$.
The isomorphism here is meant to be an isomorphism of {\em abstract} objects: in general, there is no
such isomorphism respecting the embeddings into $M$.
Later on, when we speak of ``isomorphism of summands'' of a motive,
we always mean an isomorphism of abstract motives between the summands.

We say that the {\em Krull-Schmidt principle} holds for a given object of a given  additive category,
if every direct sum decomposition of the object can be refined to a complete one (in particular,
a complete decomposition exists) and there is only one (up to a permutation of the summands) complete decomposition of the object.
In the sequel, we are constantly using the following statement which is an immediate
consequence of
Lemmas \ref{exists} and \ref{asm}
and Theorem \ref{bass}:

\begin{cor}[{cf.\! \cite[Corollary 35]{MR2264459}}]
\label{krull-schmidt}
Assume that the coefficient ring $\Lambda$ is finite.
The Krull-Schmidt principle holds for any shift of any summand of the motive of any
geometrically split $F$-variety satisfying the nilpotence principle.
In other words,
the Krull-Schmidt principle holds for the objects of the pseudo-abelian Tate subcategory in
$\CM(F,\Lambda)$
generated by the motives of the geometrically split $F$-varieties satisfying the
nilpotence principle.
\qed
\end{cor}

\begin{rem}
\label{reduced motives}
Replacing the Chow
groups $\Ch(-)$ by the {\em reduced} Chow groups
$\BCh(-)$ (cf.\! \cite[\S72]{EKM})
in the definition of the category $\CM(F,\Lambda)$,
we get a ``simplified'' motivic category $\BCM(F,\Lambda)$
(which is still sufficient for the main purpose of this paper).
Working within this category, we do not need the nilpotence principle any more.
In particular, the Krull-Schmidt principle holds (with a simpler proof)
for the shifts of the summands of
the motives of the geometrically split $F$-varieties.
\end{rem}

\subsection
{Upper, lower, and outer summands}

We assume here that the coefficient ring $\Lambda$ is connected.
We shall often assume that $\Lambda$ is finite.

Given a correspondence,
an element $\alpha\in\Ch_{\dim X}(X\times Y)$ of the Chow group of the product of smooth complete irreducible varieties
$X$ and $Y$, we write $\mult\alpha\in\Lambda$ for the {\em multiplicity} (or {\em multiplicity over the first factor})
of $\alpha$,
\cite[definition of \S75]{EKM}.
Multiplicity of a composition of two correspondences is the product of multiplicities of the composed
correspondences (cf.\! \cite[Corollary 1.7]{MR1751923}).
In particular, the multiplicity of a projector is idempotent and therefore $\in\{0,1\}$ because the coefficient
ring $\Lambda$ is connected.

\begin{lemma}
\label{lemma left outer}
Let  $X$ be a smooth complete irreducible variety.
Let $M$ be a summand of the motive of $X$ and let
$\pi\in\Ch_{\dim X}(X\times X)$ be the projector giving $M$.
The following four conditions on $M$ are equivalent:
\begin{itemize}
\item
$\Ch^0(M)\ne 0$;
\item
the summand $\Ch^0(M)$ of  the group ($\Lambda$-module) $\Ch^0(X)$ coincides with
the whole $\Ch^0(X)$;
\item
$\mult\pi\ne0$;
\item
$\mult\pi=1$.
\end{itemize}
\end{lemma}

\begin{proof}
The group $\Ch^0(M)$,
defined as $\Hom(M,\Lambda)$,
is the image of the endomorphism
of the group $\Ch^0(X)=\Lambda\cdot[X]$ given by the multiplication by $\mult \pi$.
\end{proof}

The dual version of Lemma \ref{lemma left outer} reads as follows:

\begin{lemma}
\label{lemma right outer}
Let $X$ be a smooth complete irreducible variety.
Let $M$ be a summand of the motive of $X$ and let
$\pi\in\Ch_d(X\times X)$, where $d=\dim X$, be the projector giving $M$.
The following four conditions on $M$ are equivalent:
\begin{itemize}
\item
$\Ch_d(M)\ne 0$;
\item
the summand $\Ch_d(M)$ of  the group ($\Lambda$-module) $\Ch_d(X)$ coincides with
the whole $\Ch_d(X)$;
\item
$\mult\pi^t\ne0$;
\item
$\mult\pi^t=1$,
\end{itemize}
where $\pi^t$ is the transpose of $\pi$.
\end{lemma}

\begin{proof}
The group $\Ch_d(M)$,
defined as $\Hom(\Lambda(d),M)$,
is the image of the endomorphism
of the same group
$\Ch_d(X)=\Ch^0(X)=\Lambda\cdot[X]$ given by the multiplication by $\mult \pi^t$.
\end{proof}

The following definition is extending some terminology (concerning the summands of the motives of quadrics) of
\cite{Vishik-IMQ}.

\begin{dfn}
\label{def-outer}
Let $M\in\CM(F,\Lambda)$ be a summand of the motive of a smooth complete
irreducible variety.
The summand $M$ is called {\em upper}, if it satisfies the four equivalent conditions of Lemma \ref{lemma left outer}.
The summand $M$ is called {\em lower}, if it satisfies the four equivalent conditions of Lemma \ref{lemma right outer}.
The summand $M$ is called {\em outer}, if it is upper and lower simultaneously.
\end{dfn}

For instance, the whole motive of a smooth complete irreducible
variety is an outer summand of itself.
Of course, a projector $\pi$ determines an upper summand if and only if the summand given by $\pi^t$ is lower.
Since the multiplicity does not change under extensions of scalars, a summand $M$ is upper (or lower, or outer)
if and only if the summand $M_E$ is so for some field extension $E/F$.

\begin{lemma}
\label{upper and 0-cycles}
Let $X$ be a smooth complete irreducible variety such that the degree homomorphism
$\deg:\Ch_0(X)\to\Lambda$ is an isomorphism.
The following three conditions on
a summand $M$ of $M(X)\in\CM(F,\Lambda)$ are equivalent:
\begin{itemize}
\item
$M$ is upper;
\item
the summand $\Ch_0(M)$ of the group ($\Lambda$-module)
$\Ch_0(X)$ coincides with the whole $\Ch_0(X)$;
\item
$\Ch_0(M)\ne0$.
\end{itemize}
\end{lemma}

\begin{proof}
Let $\pi\in\Ch_{\dim X}(X\times X)$ be the projector giving $M$.
Let $x$ be an element of $\Ch_0(X)$ with $\deg x=1\in\Lambda$.
In particular, $x$ is a generator of the $\Lambda$-module $\Ch_0(X)$.
Then the $\Lambda$-module $\Ch_0(M)=\Hom(\Lambda,M)$ coincides with the submodule of
$\Ch_0(X)$ generated by $(\mult\pi)\cdot x$.
\end{proof}

Replacing ``upper'' by ``lower'' and $\Ch_0$ by $\Ch^{\dim X}$ (in all 5 appearances)
in Lemma \ref{upper and 0-cycles}, we get
the dual version of it.

To characterize the upper (or lower) summands in the split case
is very easy:

\begin{lemma}
\label{tate outer}
Assume that a summand $M$ of the motive of a smooth complete irreducible
variety of dimension $d$ decomposes into a sum of Tate motives.
Then $M$ is upper if and only if the Tate motive $\Lambda$ is present in the
decomposition;
it is lower if and only if the Tate motive $\Lambda(d)$ is present in the decomposition.
\end{lemma}

\begin{proof}
For any $i\in\Z$ we have:
$\Ch^0\big(\Lambda(i)\big)\ne0$ if and only if $i=0$;
$\Ch_d\big(\Lambda(i)\big)\ne0$ if and only if $i=d$.
\end{proof}

\begin{rem}
\label{outer unique}
Assume that the coefficient ring $\Lambda$ is finite.
Let $X$ be an irreducible  geometrically split
variety satisfying the nilpotence principle.
Then a complete motivic decomposition of $X$ contains precisely one upper
summand and it follows by
Corollary \ref{krull-schmidt}
that an upper indecomposable summand of $M(X)$ is unique up to
an isomorphism (of motives, not of summands).
Of course, the same is true for the lower summands.
\end{rem}

\begin{lemma}
\label{variacija}
Assume that the coefficient ring is finite.
Let $X$ be an irreducible geometrically split variety satisfying
the nilpotence principle.
Let $M$ be a motive.
Assume that there exist morphisms $\alpha:M(X)\to M$ and $\beta:M\to M(X)$ such that
$\mult(\beta\compose\alpha)=1$.
Then the indecomposable upper summand of $M(X)$ is isomorphic to a summand of $M$.
\end{lemma}

\begin{proof}
By Corollary \ref{cor-finite}, the composition
$\pi:=(\beta\compose\alpha)^{\compose n}\in\End M(X)$
is a projector for some integer $n\geq1$.
Therefore $\tau:=(\alpha\compose\beta)^{\compose 2n}\in\End M$ is also a projector and
the (upper) summand $(X,\pi)$ of $M(X)$
is isomorphic to the summand $(M,\tau)$ of $M$ given by the image of $\tau$:
mutually inverse isomorphisms are, say,
$$
\alpha\compose(\beta\compose\alpha)^{\compose(2n)}:(X,\pi)\to(M,\tau)\;\text{ and }\;
\beta\compose(\alpha\compose\beta)^{\compose(4n-1)}:(M,\tau)\to(X,\pi)\;.
$$
Consequently, any (in particular, the indecomposable upper) summand
of $(X,\pi)$ is isomorphic to a summand of $M$.
\end{proof}

\begin{cor}
\label{criter out iso}
Let $X$ and $Y$ be irreducible geometrically split varieties satisfying the nilpotence principle.
The indecomposable upper summands of $M(X)$ and $M(Y)$  are isomorphic if and only if
there exist multiplicity $1$ correspondences $\alpha:M(X)\to M(Y)$ and $\beta:M(Y)\to M(X)$.
\end{cor}

\begin{proof}
The ``if''-part follows by Lemma \ref{variacija}.
For the ``only if''-part simply note that if $M_X\to M_Y$ is an isomorphism of upper summands of
$M(X)$ and $M(Y)$, then the multiplicity of the composition $M(X)\to M_X\to M_Y\to M(Y)$
is an invertible element of $\Lambda$.
\end{proof}

\subsection
{Rank of a motive}

We are still assuming that the coefficient ring $\Lambda$ is connected.

\begin{dfn}
Let $M$ be a geometrically split motive.
Over an extension of the base field the motive
$M$ becomes isomorphic to a finite sum of Tate motives.
The {\em rank} $\rk M$ of $M$ is defined as the number of
the summands in this decomposition.
\end{dfn}

\begin{rem}
The number of the summands in the above definition does not depend on the choice of
the extension or of the decomposition.
This is simply the rank of the free $\Lambda$-module
$\colim_{L/F}\Ch_*(M_L)$, where the colimit is taken over all field
extensions $L/F$.
\end{rem}

\begin{example}
\label{rank grassmannian}
Let $n$ be a positive integer and let $i$ be an integer in the interval $[0,\; n]$.
Let $A$ be a central simple $F$-algebra of degree $n$.
Since the variety $X=X(i,A)$ (see \S\ref{Varieties (of flags) of ideals}) is a twisted form of the grassmannian of
$i$-planes in an $n$-dimensional vector space, the rank of the motive of $X$ coincides with the rank of the motive of the grassmannian
and is equal to the binomial coefficient $\binom{n}{i}$.
\end{example}

\begin{rem}
Let $M$ be a direct summand of a geometrically split variety $X$ satisfying the
nilpotence principle.
Then we have: $\rk M=0$ if and only if $M=0$;
$\rk M=\rk M(X)$ if and only if $M=M(X)$.
\end{rem}

We are going to assume that the connected coefficient ring $\Lambda$ is finite.
Any non-invertible element of $\Lambda$ is then nilpotent (e.g., by \ref{finite}) and therefore
$\Lambda$ is local (e.g. by Lemma \ref{asm}).
Its residue field is also finite, and we write $p$ for the prime integer which is the characteristic of this field.

\begin{rem}
Let $k$ be the residue field of $\Lambda$.
The change of coefficients functor $\CM(F,\Lambda)\to\CM(F,k)$ induces bijection on the isomorphism classes of the objects,
\cite[Corollary 2.6]{MR2377113}.
Therefore all results below concerning the motives with coefficients in a finite connected ring,
are essentially about the motives with coefficients in a finite field.
Although $\F_p$ as the coefficient ring is sufficient for the current applications, a finite extension of $\F_p$ might be
also of interest.
Since the proofs for any $\Lambda$ do not differ much from the proofs for $\F_p$, we stay with the arbitrary finite connected coefficient ring.
\end{rem}

For any integer $l\ne0$,
we write $v_p(l)$ for the exponent of the highest power of $p$ dividing $l$ (and $v_p(0):=+\infty$).

\begin{lemma}
\label{rank-degree}
Assume the connected coefficient ring $\Lambda$ is finite.
Let $M$ be a direct summand of the motive (with coefficients in $\Lambda$)
of a geometrically split variety $X$.
Let $d$ be the greatest common divisor
of the degrees of the closed points on $X$.
Then $v_p(d)\leq v_p(\rk M)$.
\end{lemma}

\begin{proof}
We may assume that $X$ is irreducible.

The residue field of the local ring $\Lambda$ is the finite field $\F_q$, where $q$ is a power of $p$.
Since changing coefficients does not change the rank of a motive,
we may assume (changing the coefficients via the homomorphism $\Lambda\to\F_q$) that $\Lambda=\F_q$.
Let $f\in\Z[t]$ be a monic integral  polynomial in a variable $t$ such that
$\F_q\simeq(\Z/p)[t]/(f)$.
Let $n=v_p(d)$.
Let $f'\in(\Z/p^n)[t]$ be a monic lifting of $f$.
The ring $\Lambda'=(\Z/p^n)[t]/(f')$ is also connected and finite, and there is an epimorphism
$\Lambda'\to\F_q$.
The image of an integer $m$ in $\Lambda'$ is the degree of a $0$-cycle class in
$\CH_0(X)\otimes\Lambda'$ only if $v_p(m)\geq n$ (that is, $m=0\in\Lambda'$).

Let $\pi\in\End M(X)=\Ch_{\dim X}(X\times X)=\CH_{\dim X}(X\times X)\otimes\F_q$
be the projector defining the summand $M$.
Let $\pi'\in\CH_{\dim X}(X\times X)\otimes\Lambda'$ be a lifting of $\pi$.
By Lemma \ref{finite},
replacing $\pi'$ by its appropriate power, we may assume that
$\pi'$ is a projector.

The rank of the motive $(X,\pi')\in\CM(F,\Lambda')$ coincides with $\rk M$.
Let $L/F$ be a splitting field of the motive $(X,\pi')$.
Mutually inverse isomorphisms between $(X,\pi')_L$ and a sum of $m=\rk M$
Tate motives are given by two sequences of homogeneous elements
$a_1,\dots,a_m$ and $b_1,\dots,b_m$ in $\CH(X_L)\otimes\Lambda'$ satisfying
$\pi'_L=a_1\times b_1+\dots+a_m\times b_m$ and such that for any $i,j=1,\dots,m$
the degree $\deg(a_ib_j)\in\Lambda'$ is $0$ for $i\ne j$ and $1$ for $i=j$.
The pull-back of $\pi'$ via the diagonal morphism of $X$ is therefore a $0$-cycle class
on $X$ of degree $m$.
It follows that $v_p(m)\geq n$, that is, $v_p(d)\leq v_p(\rk M)$.
\end{proof}

Lemma \ref{rank-degree} gives a new, particularly simple proof of the following fact
(the original proof, given in \cite{MR1356536},
makes use of Quillen's computation of $K$-theory of
$X$):

\begin{cor}
\label{SB-cor}
The motive with coefficients in a finite connected ring $\Lambda$ of the Severi-Brauer variety $X$ of a
central division algebra of degree $p^n$ ($n\geq0$ an integer, $p$ the characteristic of the residue field of $\Lambda$)
is indecomposable.
\end{cor}

\begin{proof}
Since $\rk M(X)=p^n$ and the $\gcd$ of the degrees of the
closed points on $X$ is also $p^n$, the rank of any summand of $M(X)$
is $0$ or $p^n$ by Lemma \ref{rank-degree}.
\end{proof}

\subsection
{Varieties (of flags) of ideals}
\label{Varieties (of flags) of ideals}

Let $A$ be a central simple $F$-algebra.
The $F$-dimension of any right ideal in $A$ is divisible by the degree
$\deg A$ of $A$;
the quotient is the {\em reduced dimension} of the ideal.
For any integer $i$, we write $\SB(i;A)$
for the generalized Severi-Brauer variety of the right ideals in $A$ of reduced
dimension $i$.
In particular, $\SB(0;A)=\Spec F=\SB(\deg A;A)$ and
$\SB(i;A)=\emptyset$ for $i$ outside of the interval $[0,\;\deg A]$.
The variety $\SB(1;A)$ is the usual Severi-Brauer variety of $A$ studied in \cite{MR657430}.

For a finite sequence of integers $i_1,\dots,i_r$, we write
$\SB(i_1,\dots, i_r;A)$
for the variety of flags of right ideals in $A$ of reduced
dimensions $i_1,\dots,i_r$ (non-empty if and only if $0\leq i_1\leq\dots\leq i_r\leq\deg A$).

We have
$
\ind A_{F\left(X(i_1,\dots,i_r;A)\right)}=
\gcd(i_1,\dots,i_r,\ind A).
$
Another classical property of the variety $X(i_1,\dots,i_r;A)$
which we are using frequently is that the greatest common divisor of the degrees of its closed points is equal to
$
\ind A/\gcd(i_1,\dots,i_r,\ind A).
$

The varieties introduced above are projective homogeneous under the
natural action of the algebraic group $\Aut A$.
As in \S\ref{Introduction},
we write $\cX_A$ for the class $\cX_{\Aut A}$
of all finite direct products of such varieties.

\subsection
{Canonical dimension}
\label{Canonical dimension}

The notion of canonical dimension was introduced in \cite{MR2183253},
of canonical dimension at $p$ in \cite{MR2258262}.
We refer to \cite{MR2258262} and \cite{ASM-ed} for proofs of the statements cited below.

Let $X$ be a smooth complete irreducible variety over $F$.
{\em Canonical dimension} $\cd X$ of $X$ is defined  as the
least dimension of the image of a rational map $X\RatM X$.
For a positive prime integer $p$, {\em canonical dimension at $p$}, or {\em canonical $p$-dimension}
$\cd_p X$ of $X$ is defined as the
least dimension of the image of a morphism $X'\to X$,
where $X'$ is an irreducible variety with $\dim X'=\dim X$ admitting a dominant morphism $X'\to X$ of
a $p$-coprime degree.

If $L/F$ is a finite field extension of degree prime to $p$, then $\cd_p X_L=\cd_p
X$.
If two smooth complete irreducible $F$-varieties $X_1$ and $X_2$ are such that there exist
rational maps $X_1\RatM X_2$ and $X_2\RatM X_1$, then $\cd X_1=\cd X_2$ and $\cd_p X_1=\cd_p X_2$
for any $p$.

One has $\cd_p X\leq\cd X\leq \dim X$.
The variety $X$ is called {\em incompressible} ({\em minimal} in \cite{ASM-ed})
if $\cd X=\dim X$
and {\em $p$-incompressible} ({\em $p$-minimal} in \cite{ASM-ed}) if $\cd_p X=\dim X$.
A $p$-incompressible (for some prime $p$) variety is incompressible.

The variety  $X$ is $p$-incompressible if for any element
$\alpha\in\CH_{\dim X}(X\times X)$ the multiplicity  $\mult(\alpha)$ coincides modulo $p$ with
the multiplicity $\mult(\alpha^t)$ of the transpose $\alpha^t$ of $\alpha$.
One can show that the converse statement holds for homogeneous $X$ (but we will not use this in the paper).

\begin{lemma}
\label{outer and incompressible}
A smooth complete irreducible
variety $X$ is $p$-incompressible if any upper summand of the motive
of $X$ with coefficients in $\F_p$ is outer.
\end{lemma}

\begin{proof}
We assume that there is an element $\alpha\in\CH_{\dim X}(X\times X)/p$ with
$\mult(\alpha)\ne\mult(\alpha^t)$.
Replacing $\alpha$ by $\alpha-\mult(\alpha^t)\cdot\Delta_X$, where $\Delta_X$ is the diagonal class,
we get $\mult(\alpha)\ne0$ and $\mult(\alpha^t)=0$.
A power
of the correspondence $\alpha$ is a projector which determines a summand of $M(X)$.
This summand is upper but not lower.
\end{proof}

\begin{cor}
\label{indecomposable-incompressible}
If the motive with coefficients in $\F_p$ of a smooth complete variety $X$ is
indecomposable, then the variety $X$ is $p$-incompressible.
\qed
\end{cor}

In Lemma \ref{outer and incompressible} and in Corollary \ref{indecomposable-incompressible},
$\F_p$ can be replaced by any finite connected coefficient ring with residue field of characteristic $p$.

\section
{Motivic decomposition theorems}

The coefficient ring $\Lambda$ is supposed to be finite and connected in this section.
We write $p$ for the prime integer which is the
characteristic of the residue field of $\Lambda$.

Let $G$ be a semisimple affine algebraic group of inner type over a field $F$.
We write ${\cT}_G$ (or simply $T$)
 for the set of the conjugacy classes of the {\em maximal} parabolic subgroups in $\bar{G}$
 ($T_G$ can be identified with the set of vertices of the Dynkin diagram of $G$).
The subsets of $\cT$ are in natural one-to-one correspondence with the set of conjugacy classes of all parabolic
subgroups in $\bar{G}$.
This correspondence is defined as follows:
the conjugacy class corresponding $\tau\subset T$ is the image of the map of the set of Borel subgroups in $\bar{G}$
into the set of parabolic subgroups in $\bar{G}$, associating to a Borel subgroup $B$
the intersection of representatives containing $B$ of the classes in $\tau$.
For any subset $\tau\subset \cT$, we write $X_\tau$ or $X_{\tau,G}$
for the projective $G$-homogeneous $F$-variety of parabolic subgroups in $G$
of the type $\tau$.
For instance, $X_{\cT}$ is the variety of Borel subgroups.
Any projective $G$-homogeneous variety is isomorphic to $X_\tau$ for some $\tau$.

\begin{dfn}
Let $G$ and $\cT$ be as above.
For any $\tau\subset \cT$, we write $M_\tau$ or $M_{\tau,G}$ for
the indecomposable upper (see Definition \ref{def-outer} and Remark \ref{outer unique}) summand of
the Chow motive (with coefficients in $\Lambda$) of the variety $X_\tau$.
The set of the isomorphism classes $[M_\tau]$, where $\tau$ runs over the subsets in $\cT$, is called the
{\em set of upper motives}
of the algebraic group $G$.
\end{dfn}

\begin{rem}
For a field extension $E/F$, the motive $M_{\tau,G_E}$ is a summand of the motive
$(M_{\tau,G})_E$ and $M_{\tau,G_E}\ne(M_{\tau,G})_E$ in general.
\end{rem}

\begin{rem}
\label{rem-outiso}
The motives $M_{\tau,G}$ and $M_{\tau',G'}$ are isomorphic for certain  $G=G'$ and certain $\tau\ne\tau'$.
They are also isomorphic for certain $G\not\simeq G'$ and certain $\tau,\tau'$ (see Example \ref{G and G'}).
A simple criterion of isomorphism is given in  Corollary \ref{criter out iso}.
\end{rem}

\begin{example}
\label{G and G'}
Let $G'$ be the semisimple anisotropic kernel of the group $G$.
The set $\cT_{G'}$ is canonically identified with the subset of $\cT_G$
consisting of the classes of those maximal parabolic subgroups in $\bar{G}$ which are not defined over $F$
(see \cite{MR0224710}).
For any $\tau\subset\cT_G$, the motive $M_{\tau,G}$ is isomorphic to
the motive $M_{\tau',G'}$ where $\tau'=\tau\cap\cT_{G'}$.
In particular, the set of upper motives of $G$ is determined by $G'$.
\end{example}

For any field extension $E/F$ let $\tau_E$ and also $\tau_{E,G}$ stand for the subset of $\cT_G$
consisting of the classes of those maximal parabolic
subgroups in $\bar{G}$ which are defined over $E$.
For instance, for any $\tau\subset\cT_G$ and  $E=F(X_\tau)$ we have $\tau\subset\tau_E$.
Note that $\tau_F=\cT_G\setminus\cT_{G'}$.

\begin{thm}[General motivic decomposition theorem]
\label{Gbasic1}
Let $G$ be a semisimple affine algebraic group of inner type over $F$.
Any indecomposable summand of the Chow motive (with coefficients in a finite connected ring) of
any variety $X$ in the class $\cX_G$ (see \S\ref{Introduction})
is isomorphic to a shift of $M_\tau$ for some $\tau\subset T_G$ satisfying $\tau\supset\tau_{F(X)}$.
\end{thm}

\begin{example}
Assume that  the variety $X$ in Theorem \ref{Gbasic1} is  {\em generically split}, that is, $\tau_{F(X)}=\cT$.
Then by Theorem \ref{Gbasic1} the complete motivic decomposition of $X$ consists of shifts of the motive
$M_\cT$.
This is the decomposition constructed in \cite[Theorem 5.17]{J-invariant}.
\end{example}

\begin{proof}[Proof of Theorem \ref{Gbasic1}]
We proof
Theorem \ref{Gbasic1}
simultaneously for all $G$ (over all fields) using an induction on $n=\rk G$
(the rank of $G$ which is the number of elements in $T_G$).
The base of the induction is $n=0$ where $\cX_G=\{\Spec F\}$ and the statement is trivial.

From now on we are assuming that $n\geq1$ and that Theorem
\ref{Gbasic1} is already proven for all groups of rank $<n$.

Let $X\in\cX_G$ and let $M$ be an indecomposable summand of $M(X)$.
We have to show that $M$ is isomorphic to a shift of $M_\tau$ for some $\tau\subset T_G$ containing
$\tau_{F(X)}$.

We may assume that $X\ne\Spec F$.
In this case the set $\tau_{L,G}$, where $L=F(X)$, is non-empty.
Let $G'/L$ be the semisimple anisotropic kernel of the group $G_L$.
We recall that the set $\cT_{G'}$ is identified with $\cT_G\setminus\tau_{L,G}$.
In particular, $\rk G'<n$.
According to \cite{MR2178658},
the motive of $X_L$ decomposes into a sum of shifts of motives
of products of projective $G'$-homogeneous $L$-varieties.
It follows by the induction hypothesis (applied to $G'$),
that each summand of the complete motivic decomposition of $X_L$ is a shift of
$M_{\tau',G'}$ for some $\tau'\subset\cT_{G'}$.
The complete decomposition of $M_L$ is a part of the above decomposition
(in the sense of the Krull-Schmidt principle, see \S\ref{Chow motives with finite coefficients}).

Each summand of the complete decomposition of $M_L$
decomposes over an algebraic closure $\bar{L}$ of $L$ into a sum of
Tate motives.
This gives a decomposition of $\bar{M}=M_{\bar{L}}$ into a sum of
Tate motives.
Let us choose a Tate summand $\Lambda(i)$ with the smallest $i$ in the decomposition of $\bar{M}$.
This summand comes from the decomposition of the $i$th shift of
some $\bar{M}_{\tau',G'}$ for some $\tau'\subset\cT_{G'}$.
We set $\tau=\tau'\cup\tau_L\subset\cT_G$.
We shall show that $M\simeq M_{\tau,G}(i)$
for these $\tau$ and $i$.

We write $Y$ for the $F$-variety $X_{\tau,G}$ and we write
$Y'$ for the $L$-variety $X_{\tau',G'}$.
We write $N$ for the $F$-motive $M_{\tau,G}$ and we write
$N'$ for the $L$-motive $M_{\tau',G'}$.

By Lemma \ref{variacija} and since $M$ is indecomposable,
it suffices to construct morphisms
$$
\alpha:M(Y)(i)\to M\;\;\text{and}\;\;
\beta:M\to M(Y)(i)
$$
satisfying $\mult(\beta\compose\alpha)=1$.

The fixed above summand $\Lambda(i)$ of $\bar{M}$ is produced
by two elements
$$
a\in\Ch_i(\bar{X})=\Hom\big(\Lambda(i),M(\bar{X})\big)
\text{ and }
b\in\Ch^i(\bar{X})=\Hom\big(M(\bar{X}),\Lambda(i)\big)
$$
such that the degree of the $0$-cycle class $(a\cdot b)\in\Ch_0(\bar{X})$ is $1\in\Lambda$
(the condition on the degree means that the composition
$$
\begin{CD}
\Lambda(i)@>a>> M(\bar{X})@>b>>\Lambda(i)
\end{CD}
$$
is the identity).
The element $a$ is the image under the embedding
$\Ch_0(\bar{N}')\hookrightarrow\Ch_i(\bar{X})$
of the class of a rational point in $\Ch_0(\bar{Y}')=\Ch_0(\bar{N}')$
(see Lemma \ref{upper and 0-cycles}).
Since $\tau'\subset\tau$, the $L$-variety $Y'$ has an $L(Y)$-point, and it follows that
the element $a_{\bar{L}(Y)}$ is $L(Y)$-rational
(that is, lies in the image of the change of field homomorphism of the Chow groups).
Since $\tau_L\subset\tau$, the product $X$ of projective homogeneous varieties has an $F(Y)$-point
and therefore the field extension $L(Y)/F(Y)$ is purely transcendental.
Consequently, the element $a_{\bar{L}(Y)}$ is $F(Y)$-rational and
lifts to an element $\alpha_1\in\Ch_{\dim Y+i}(Y\times X)$.
We mean here a lifting with respect to the composition
$$
\begin{CD}
\Ch_{\dim Y+i}(Y\times X) \onto \Ch_i(X_{F(Y)}) @>{\res_{\bar{L}(Y)/F(Y)}}>> \Ch_i(X_{\bar{L}(Y)})
\end{CD}
$$
where the first map is the epimorphism given by the pull-back with respect to the morphism
$X_{F(Y)}\to Y\times X$ induced by the generic point of the variety $Y$.

We define the morphism $\alpha$ as the composition
$$
\begin{CD}
M(Y)(i)@>{\alpha_1}>> M(X)@>>> M
\end{CD}
$$
where the second map is the projection of $M(X)$ onto its summand $M$.

Note that the element $b$ is the image of the class $[\bar{Y}']\in\Ch^0(\bar{Y}')=\Ch^0(\bar{N}')$
under the embedding $\Ch^0(\bar{N}')\hookrightarrow\Ch^i(\bar{X})$.
Let $\beta_1\in\Ch_{\dim Y'}(Y'\times Y_L)$ be the class of
the closure of the graph of a rational map (of $L$-varieties)
$Y'\RatM Y_L$ (which exists because $\tau\subset\tau_L\cup \tau'$).
Let $\beta_2$ be the image of $\beta_1$ under the composition
of the homomorphisms
\begin{multline*}
\Ch^{\dim Y}(Y'\times Y_L)=\Ch^{\dim Y}\big(M(Y')\otimes M(Y_L)\big)\to
\Ch^{\dim Y}\big(N'\otimes M(Y_L)\big)\to \\
\Ch^{\dim Y+i}\big(M(X_L)\otimes M(Y_L)\big)=\Ch^{\dim Y+i}(X_L\times Y_L)
\end{multline*}
where the first arrow is induced by the projection $M(Y')\to N'$ and the second arrow induced
by the imbedding $N'(i)\to M(X_L)$.
The element
$\beta_2$
lifts to an element
$$
\beta_3\in\Ch^{\dim Y+i}(X\times X\times Y).
$$
We mean here a lifting with respect to the
epimorphism
$$
\begin{CD}
\Ch^{\dim Y+i}(X\times X\times Y) \onto \Ch^{\dim Y+i}\big((X\times Y)_L\big)
\end{CD}
$$
given by the pull-back with respect to the morphism
$X\times X\times Y\to (X\times Y)_L$ induced by the generic point of the second factor in this
triple direct product.

Let $\pi\in\Ch_{\dim X}(X\times X)$ be the projector defining the
summand $M$ of $M(X)$.
Considering $\beta_3$ as a correspondence from $X$ to $X\times Y$, we define
$$
\beta_4\in\Ch^{\dim Y+i}(X\times X\times Y)
$$
as the
composition $\beta_3\compose\pi$.
We get
$$
\beta_5\in\Ch^{\dim Y+i}(X\times Y)=
\Ch_{\dim X-i}(X\times Y)
$$
as the image
of $\beta_4$ under the pull-back with respect to the diagonal
of $X$.
Finally, we define the morphism $\beta$ as the
composition
$$
\begin{CD}
M@>>>M(X)@>{\beta_5}>> M(Y)(i).
\end{CD}
$$

We finish the proof by checking that $\mult(\beta\compose\alpha)=1$.
Since the multiplicity is not changed under extension of scalars, we may do the computation over the field
$\bar{L}$.
Decompositions of the motives of
the varieties $\bar{Y}$ and $\bar{X}$
into sums of Tate motives give certain homogeneous $\Lambda$-bases of their Chow groups with coefficients
in $\Lambda$ (a Tate summand $\Lambda(j)$ of, say, $M(\bar{X})$ gives the basis element
$\Lambda(j)\to M(\bar{X})\in\Ch_j(\bar{X})$).
We may use here an arbitrary motivic decomposition of $\bar{Y}$.
As to $\bar{X}$, let us use a motivic decomposition
which is obtained by taking the chosen above decomposition of $\bar{M}$ and some decomposition of the complementary to
$\bar{M}$ motivic summand of $\bar{X}$.

The basis of $\Ch(\bar{Y})$ contains (a multiple with an invertible coefficient of)
the class $1=[\bar{Y}]$ and (a multiple with an invertible coefficient of) the class $x$ of a rational point.
The element $a$ is in the basis of $\Ch(\bar{X})$.
As to the element $b$, it is in the {\em dual} basis of $\Ch(\bar{X})$, where ``dual'' means dual with respect to the bilinear
form $\Ch(\bar{X})\times\Ch(\bar{X})\to\Lambda$, $(x_1,x_2)\mapsto \deg(x_1\cdot x_2)$.

We consider the Chow group of the products
$\bar{Y}\times \bar{X}$ and $\bar{X}\times \bar{X}$
together with the bases given by the external products of the elements of the bases
of the Chow groups of the factors, where in the case of $\bar{X}\times \bar{X}$
we are using the {\em dual} basis for the first factor and we are using the
``original'' basis for the second factor.
We have $\bar{\alpha}_1=1\times a+\dots$, where ``$\dots$'' stands for a linear combination of only those basis elements
whose first factor has codimension $>\codim 1=0$.
The projector $\pi$ which determines the summand $M$ of the motive of $X$ looks over $\bar{L}$
as $\bar{\pi}=b\times a+\dots$, where ``$\dots$'' stands for a linear combination of
basis elements $b'\times a'$ of $\Ch(\bar{X}\times\bar{X})$ satisfying $b'\ne b$, $a'\ne a$, and $\codim b'\geq\codim b=i$.
Since $\alpha=\pi\compose\alpha_1$, it follows that
$\bar{\alpha}=1\times a+\dots$, where ``$\dots$'' stands
for a linear combination of only those basis elements whose first factor is of positive codimension.

Now let us go through the construction of $\beta$.
To describe $\bar{\beta}_1$, we fix a homogeneous basis of $\Ch(\bar{Y}')$ and use it to build up a basis of
$\Ch(\bar{Y}'\times \bar{Y})$.
Abusing notation we write $1$  also for the unit class in $\Ch(\bar{Y}')$.

We have $\bar{\beta}_1=1\times x+\dots$
where ``$\dots$'' stands
for a linear combination of only those basis elements whose first factor is of positive codimension.
Then we have $\bar{\beta}_2=b\times x+\dots$,
where we are using the basis of
$\Ch(\bar{X}\times \bar{Y})$ obtained out of the {\em dual} basis of $\Ch(\bar{X})$, and
where ``$\dots$'' stands
for a linear combination of only those basis elements whose first factor is of codimension $>\codim b=i$.
For $\beta_3$ we have
$\bar{\beta}_3=b\times1\times x+\dots$, where $1$ is the unit of $\Ch(\bar{X})$ and ``$\dots$'' stands for a linear combination
of basis elements of $\Ch(\bar{X}\times \bar{X}\times\bar{Y})$
which have a second factor of positive codimension {\em or} the first factor of dimension $>\codim b=i$.
For this triple product, we are using the basis obtained out of the dual basis for the first factor and, say,
the original basis for the second factor (the choice of a basis for the second factor is not important).

The element $\bar{\beta}_4$ has the same shape with the additional property that the codimension of the
first factor in each basis element
appearing in the linear combination is $\geq i$.
By this reason, $\bar{\beta}_5$ and also $\bar{\beta}$ look as
$b\times x+\dots$, where ``$\dots$'' stands
for a linear combination of only those basis elements whose first factor is of codimension $>\codim b=i$.
Therefore $\bar{\beta}\compose\bar{\alpha}=1\times x+\dots$, where ``$\dots$'' stands
for a linear combination of only those basis elements whose first factor is of positive codimension.
It follows that
$\mult(\bar{\beta}\compose\bar{\alpha})=1$.
\end{proof}

Let us now translate the statement of Theorem \ref{Gbasic1} in the case where $G$ is of inner type $\cat{A}$.

\begin{dfn}
Let $n$ be a non-negative integer.
Let $D$ be a central division $F$-algebra of degree $p^n$.
For any integer $l$ satisfying $0\leq l\leq n$,
we write $M_{l,D}$ for
the indecomposable upper (see Definition \ref{def-outer} and Remark \ref{outer unique}) summand of
the Chow motive (with coefficients in $\Lambda$) of the generalized Severi-Brauer
variety $X(p^l,D)$.
\end{dfn}

Note that $M_{l,D}\not\simeq M_{m,D}$ for $l\ne m$.

\begin{thm}[Type $\cat{A}$ motivic decomposition theorem]
\label{basic1}
Let $A$ be a central simple $F$-algebra.
Let $p$ be a positive prime integer.
Let $D$ be a $p$-primary division algebra Brauer-equivalent to the $p$-primary component of $A$.
Then any indecomposable summand of the Chow motive with coefficients in a finite connected ring $\Lambda$
with residue field of characteristic $p$ of
any variety $X$ in the class $\cX_A=\cX_{\Aut A}$
is isomorphic to a shift of $M_{l,D}$ for some $l\leq v_p(\ind A_{F(X)})$.
\qed
\end{thm}

\begin{proof}
Any projective ($\Aut A$)-homogeneous variety is isomorphic to a flag variety
$X:=X(i_1,\dots,i_r;A)$ for some integers $r\geq1$ and $i_1,\dots,i_r$ satisfying
$0<i_1<\dots<i_r\leq\deg A$.
Setting $l=v_p\big(\gcd(i_1,\dots,i_r,\ind A)\big)$, we have a rational map
$X\RatM X(p^l,D)$ (because $\ind D_{F(X)}$ divides $p^l$) and a multiplicity $1\in\Lambda$ correspondence $X(p^l,D)\corr X$.
By Corollary \ref{criter out iso}, the upper motivic summand of the variety $X$ is isomorphic to $M_{l,D}$.
\end{proof}

\section
{Motivic structure and incompressibility theorems}
\label{Motivic structure theorem}

We come back to an arbitrary central division $F$-algebra $D$ of degree $p^n$.
The motivic structure theorem describes some properties of the
motives $M_{m,D}$, $0\leq m\leq n$:

\begin{thm}[Motivic structure theorem]
\label{basic2}
The summand $M_{m,D}$ of the motive of the variety $X(p^m,D)$
is outer and
$v_p(\rk M_{m,D})=n-m$.
\end{thm}

\begin{proof}
We prove Theorem \ref{basic2} by induction on $n$.
The base of the induction is the case of $n=0$ which is trivial.
Below we are assuming that $n\geq1$.

The statement is trivial for $m=n$.
Below we are assuming that $m<n$.

In the case of $m=0$  we know (see Corollary
\ref{SB-cor}) that $M_{0,D}=M\big(X(1,D)\big)$ and $\rk M\big(X(1,D)\big)=p^n$.
Below we are assuming that $m\geq1$.

We ask the reader to note that in this proof we do not pay attention
to the shifting numbers of summands in motivic decompositions:
when we say that a motive $W$ is a summand of a motive $V$, we mean that a {\em shift}
of $W$ is a summand of $V$.

Let $L$ be the function field of the variety $X(p^{n-1},D)$.
Let $C$ be a central division $L$-algebra (of degree $p^{n-1}$) Brauer-equivalent
to $D_L$.
By \cite[Theorem 10.13]{MR1758562} (or also by \cite{MR2178658}),
the motive $M\big(X(p^m,D)\big)_L$ decomposes into the sum (of some shifts) of the motives
$$(i_1,\dots,i_p):=
M\big(X(i_1,C)\times X(i_2,C)\times\dots\times
X(i_p,C)\big),
$$
where $i_1,\dots,i_p$ run over the non-negative integers
satisfying $i_1+i_2+\dots+i_p=p^m$.
The upper summand in this decomposition is
$(p^m,0,\dots,0)=M\big(X(p^m,C)\big)$.

Let $M=M_{m,D}$.
Since the summand $M$ of $M\big(X(p^m,D)\big)$ is upper,
the $L$-motive $M_L$ contains the summand $M_{m,C}$ of the summand $(p^m,0,\dots,0)$.
We claim that in fact $M_L$ contains the $p$ summands $M_{m,C}$ coming from each of the $p$ summands
\begin{multline*}
(p^m,0,\dots,0)=M\big(X(p^m,C)\big),\;
(0,p^m,0,\dots,0)=M\big(X(p^m,C)\big),\;\\
\dots,\;
(0,\dots,0,p^m)=M\big(X(p^m,C)\big).
\end{multline*}
Indeed, since the degree of any closed point on the variety $X(p^m,D)$ is divisible by $p^{n-m}$,
we have $v_p(\rk M)\geq n-m$ by Lemma \ref{rank-degree}.
On the other hand, by the induction hypothesis, $v_p(\rk M_{m,C})=n-1-m$.
The remaining summands of the complete motivic decomposition of $X(p^m,D)_L$
are $M_{l,C}$ with $l\leq m-1$;
$v_p$ of their ranks are at least $n-m$
by Lemma \ref{rank-degree} (or by the induction hypothesis once again).

We have in particular proved that $M_L$ contains the summand $M_{m,C}$ coming from
the lower summand $(0,\dots,0,p^m)$ of $M\big(X(p^m,D_L)\big)$.
Since $M_{m,C}$ is a lower summand of $X(p^m,C)$ (by the induction hypothesis),
it follows that $M$ is lower.

It remains to prove the statement on the rank of $M=M_{m,D}$.

To do this, we look at ranks of the summands in the complete
decomposition of $M_L$.
We have $p$ summands $M_{m,C}$ with $v_p(\rk M_{m,C})=n-1-m$.
So, we have $v_p(\rk M_{m,C}^{\oplus p})=n-m$ for this part of the complete
decomposition of $M_L$.

The summands $M_{l,C}$ with $l\leq m-2$ have $v_p(\rk M_{l,C})\geq n-m+1$;
so, we do not care about the number of such summands.

In order to show that $v_p(\rk M)=n-m$, it suffices to show that the number of
the summands $M_{m-1,C}$ (which have
$v_p(\rk M_{m-1,C})=n-m$) in the complete decomposition of $M_L$ is divisible by $p$.

Let us first count the number of summands $M_{m-1,C}$ in the complete motivic decomposition
of $X(p^m,D)_L$.
We use the complete decomposition of $X(p^m,D)_L$ which is a refinement of the decomposition into a
sum of $(i_1,\dots,i_p)$ considered above.
There are two types of such summands $M_{m-1,C}$: those which appear as
summands of $M\big(X(p^m,C)\big)$ (if any) and all the others.
Since the number of the summands $M\big(X(p^m,C)\big)$ is $p$, the number of the summands
$M_{m-1,C}$ of the first type is divisible by $p$.

The summands of the second type are summands of the summands
$(i_1,\dots,i_p)$
with $\min_j v_p(i_j)=m-1$.
There is precisely one such summand $(i_1,\dots,i_p)$ with $i_1=\dots=i_p$, namely,
the summand $(p^{m-1},\dots,p^{m-1})$.
All the other such summands $(i_1,\dots,i_p)$ can be divided into disjoint groups which are orbits
of the action of the cyclic group $\Z/p$ on the indices:
the group containing a given $(i_1,\dots,i_p)$ consists of
$$
(i_1,\dots,i_p),\;
(i_2,\dots,i_p,i_1),\;\dots,\;
(i_p,i_1,\dots,i_{p-1})
$$
(note that these are $p$ different summands because the integer $p$ is prime).
So, the number of the summands $M_{m-1,C}$ coming from the summands $(i_1,\dots,i_p)$ different from
$(p^{m-1},\dots,p^{m-1})$ is divisible by $p$.

As to the remaining summand $(p^{m-1},\dots,p^{m-1})=
M\big(X(p^{m-1},C)^{\times p}\big)$,
we can show (using the induction hypothesis)
that the number of the summands $M_{m-1,C}$ in its complete motivic
decomposition is also divisible by $p$.
More generally, we can show
that the number of the summands $M_{m-1,C}$ in the complete motivic
decomposition of $X(p^{m-1},C)^{\times r}$ is divisible by $p$ as long as $r\geq2$.
Indeed,
$$
v_p\Big(\rk M\big(X(p^{m-1},C)^{\times r}\big)\Big)=r(n-m)>n-m
$$ (we recall that $m<n$).
The complete motivic decomposition of $X(p^{m-1},C)^{\times r}$ consists of the motives
$M_{l,C}$ with $l\leq m-1$.
Finally, $v_p(\rk M_{l,C})>n-m$ for $l<m-1$ and $v_p(\rk M_{m-1,C})=n-m$.

We have shown that the number of summands $M_{m-1,C}$ in the complete motivic
decomposition of $X(p^m,D)_L$ is divisible by $p$.
We finish by showing that the number of those summands $M_{m-1,C}$ which are not
in the complete decomposition of $M_L$ is also divisible by $p$.

If a summand $M_{m-1,C}$ is not in $M_L=(M_{m,D})_L$, then it is in
$(M_{m-1,D})_L$.
However the number of summands $M_{m-1,C}$ in $(M_{m-1,D})_L$ is
equal to $p$.
\end{proof}

Using Theorems \ref{basic1} and \ref{basic2},
we produce a new example of a generalized Severi-Brauer variety with
indecomposable motive (trivial for $n=1$, well-known for $n=2$, recently proved using
$K$-theory by Maksim Zhykhovich for $n=3$, new for $n\geq4$):

\begin{thm}
\label{2-2^n}
Let $F$ be a field.
Let $D$ be a central division $F$-algebra of degree $2^n$ with $n\geq 1$.
Then the motive with coefficients in $\F_2$ (or, more generally, in a finite connected ring with residue field
of characteristic $2$) of the variety $X(2;D)$ is indecomposable.
\end{thm}

\begin{proof}
Let us prove it by induction on $n$.
Assume that $n\geq 2$ and that the statement is already proved for algebras (over all fields)
of degree $<2^n$.
By Theorem \ref{basic1}, the motive of $X(2,D)$ is a sum of shifted copies of $M_{0,D}$ and $M_{1,D}$.
By Theorem \ref{basic2}, the upper summand $M_{1,D}$ of $M\big(X(2,D)\big)$ is outer and therefore there is no other
shifted copy of $M_{1,D}$ in the complete motivic decomposition of $X(2,D)$.
Moreover, $M_{0,D}=M\big(X(1,D)\big)$ by Corollary \ref{SB-cor}, and it follows that
if the motive of $X(2;D)$ is decomposable (for a
given $D$), then some shift of $M\big(X(1;D)\big)$ is a summand of
$M\big(X(2;D)\big)$.
We can check however that no shift of $M\big(X(1;D)\big)_L$ is a summand of
$M\big(X(2;D)\big)_L$, where $L/F$ is a field extension such that $\ind
D_L=2^{n-1}$.

Indeed, let $C$ be a central division $L$-algebra Brauer-equivalent to $D_L$.
The complete decompositions of the motives of these two varieties over $L$ are:
$$
M\big(X(1;D)\big)_L=M\big(X(1;C)\big)\oplus M\big(X(1;C)\big)(2^{n-1})=
M_{0,C}\oplus M_{0,C}(2^{n-1})
$$
and (we apply the induction hypothesis to $C$)
\begin{multline*}
M\big(X(2;D)\big)_L=\\
M\big(X(2;C)\big)\oplus
\Big(M\big(X(1;C)\big)\otimes M\big(X(1;C)\big)\Big)(2^{n-1}-1)\oplus
M\big(X(2;C)\big)(2^n)=\\
M_{1,C}\oplus
M_{0,C}(2^{n-1}-1)\oplus M_{0,C}(2^{n-1})\oplus\dots\oplus M_{0,C}(2^n-2)
\oplus
M_{1,C}(2^n),
\end{multline*}
and the motives $M_{0,C}$ and $M_{1,C}$ are not isomorphic.
\end{proof}

Theorem \ref{basic2} together with Lemma \ref{outer and incompressible} gives
the incompressibility theorem:

\begin{thm}
[Incompressibility theorem]
\label{main}
Let $p$ be a positive prime integer.
Let $n$ be a non-negative integer.
Let $F$ be a field.
Let $D$ be a central division $F$-algebra of degree $p^n$.
Let $m$ be an integer in the interval $[0,\;n]$.
Then the variety $X(p^m;D)$ is $p$-incompressible.
\qed
\end{thm}

For an arbitrary central simple $F$-algebra $A$, canonical $p$-dimension of the
varieties in $\cX_A$ is computed as follows:

\begin{cor}
\label{cor basic2}
Let $X$ be a variety in $\cX_A$.
Let $n=v_p(\ind A)$ and $m=v_p(\ind A_{F(X)})$.
Then $\cd_p X=p^m(p^n-p^m)$.
\end{cor}

\begin{proof}
Let $D$ be the $p$-primary part of a central division $F$-algebra
Brauer-equivalent to $A$.
Let $L/F$ be a finite field extension of prime to $p$ degree such that the
$L$-algebra $A_L$ is Brauer-equivalent to $D_L$.
There exist rational maps $X_L\RatM X(p^m,D_L)$ and $X(p^m,D_L)\RatM X_L$ (because each of the varieties $X_L$ and $X(p^m,D_L)$ has a rational point over the function field of the other).
It follows (cf.\!
\cite{MR2713320})
that
\begin{equation*}
\cd_p X=\cd_p X_L=\cd_p X(p^m,D_L)=\dim X(p^m,D)=p^m(p^n-p^m).
\qedhere
\end{equation*}
\end{proof}


\begin{thebibliography}{10}

\bibitem{MR657430}
{\sc Artin, M.}
\newblock Brauer-{S}everi varieties.
\newblock In {\em Brauer groups in ring theory and algebraic geometry (Wilrijk,
  1981)}, vol.~917 of {\em Lecture Notes in Math.} Springer, Berlin, 1982,
  pp.~194--210.

\bibitem{MR0249491}
{\sc Bass, H.}
\newblock {\em Algebraic {$K$}-theory}.
\newblock W. A. Benjamin, Inc., New York-Amsterdam, 1968.

\bibitem{MR2183253}
{\sc Berhuy, G., and Reichstein, Z.}
\newblock On the notion of canonical dimension for algebraic groups.
\newblock {\em Adv. Math. 198}, 1 (2005), 128--171.

\bibitem{MR2178658}
{\sc Brosnan, P.}
\newblock On motivic decompositions arising from the method of {B}ia\l
  ynicki-{B}irula.
\newblock {\em Invent. Math. 161}, 1 (2005), 91--111.

\bibitem{BRV}
{\sc Brosnan, P., Reichstein, Z., and Vistoli, A.}
\newblock Essential dimension, spinor groups, and quadratic forms.
\newblock {\em Ann. of Math. (2) 171}, 1 (2010), 533--544.

\bibitem{MR2249542}
{\sc Calm{\`e}s, B., Petrov, V., Semenov, N., and Zainoulline, K.}
\newblock Chow motives of twisted flag varieties.
\newblock {\em Compos. Math. 142}, 4 (2006), 1063--1080.

\bibitem{MR2110630}
{\sc Chernousov, V., Gille, S., and Merkurjev, A.}
\newblock Motivic decomposition of isotropic projective homogeneous varieties.
\newblock {\em Duke Math. J. 126}, 1 (2005), 137--159.

\bibitem{MR2264459}
{\sc Chernousov, V., and Merkurjev, A.}
\newblock Motivic decomposition of projective homogeneous varieties and the
  {K}rull-{S}chmidt theorem.
\newblock {\em Transform. Groups 11}, 3 (2006), 371--386.

\bibitem{EKM}
{\sc Elman, R., Karpenko, N., and Merkurjev, A.}
\newblock {\em The algebraic and geometric theory of quadratic forms}, vol.~56
  of {\em American Mathematical Society Colloquium Publications}.
\newblock American Mathematical Society, Providence, RI, 2008.

\bibitem{MR2358058}
{\sc Florence, M.}
\newblock On the essential dimension of cyclic {$p$}-groups.
\newblock {\em Invent. Math. 171}, 1 (2008), 175--189.

\bibitem{gug}
{\sc Karpenko, N.}
\newblock Unitary grassmannians.
\newblock Linear Algebraic Groups and Related Structures (preprint server) 424
  (2011, Mar 17), 17 pages.

\bibitem{isouni}
{\sc Karpenko, N., and Zhykhovich, M.}
\newblock Isotropy of unitary involutions.
\newblock arXiv:1103.5777v3 [math.AG] (29 Mar 2011), 21 pages.

\bibitem{unitary}
{\sc Karpenko, N.~A.}
\newblock Hyperbolicity of unitary involutions.
\newblock Linear Algebraic Groups and Related Structures (preprint server) 397
  (2010, Jul 30), 10 pages.

\bibitem{gog}
{\sc Karpenko, N.~A.}
\newblock Incompressibility of generic orthogonal grassmannians.
\newblock Linear Algebraic Groups and Related Structures (preprint server) 409
  (2010, Nov 23), 7 pages.

\bibitem{og}
{\sc Karpenko, N.~A.}
\newblock Incompressibility of orthogonal grassmannians.
\newblock arXiv:1107.0438v1 [math.AG] (3 Jul 2011), 5 pages.

\bibitem{qweil}
{\sc Karpenko, N.~A.}
\newblock Incompressibility of quadratic {W}eil transfer of generalized
  {S}everi-{B}rauer varieties.
\newblock Linear Algebraic Groups and Related Structures (preprint server) 362
  (2009, Oct 28), 12 pages. J. Inst. Math. Jussieu, to appear.

\bibitem{oddisotro-tignol}
{\sc Karpenko, N.~A.}
\newblock Isotropy of orthogonal involutions.
\newblock With an appendix by J.-P. Tignol. arXiv:0911.4170v3 [math.AG] (31 Jan
  2010), 13 pages. Amer. J. Math., to appear.

\bibitem{MR1356536}
{\sc Karpenko, N.~A.}
\newblock Grothendieck {C}how motives of {S}everi-{B}rauer varieties.
\newblock {\em Algebra i Analiz 7}, 4 (1995), 196--213.

\bibitem{MR1758562}
{\sc Karpenko, N.~A.}
\newblock Cohomology of relative cellular spaces and of isotropic flag
  varieties.
\newblock {\em Algebra i Analiz 12}, 1 (2000), 3--69.

\bibitem{MR1751923}
{\sc Karpenko, N.~A.}
\newblock On anisotropy of orthogonal involutions.
\newblock {\em J. Ramanujan Math. Soc. 15}, 1 (2000), 1--22.

\bibitem{hypernew-tignol}
{\sc Karpenko, N.~A.}
\newblock Hyperbolicity of orthogonal involutions.
\newblock {\em Doc. Math. Extra Volume: Andrei A. Suslin's Sixtieth Birthday\/}
  (2010), 371--392 (electronic).
\newblock With an {A}ppendix by Jean-Pierre Tignol.

\bibitem{outer}
{\sc Karpenko, N.~A.}
\newblock Upper motives of outer algebraic groups.
\newblock In {\em Quadratic forms, linear algebraic groups, and cohomology},
  vol.~18 of {\em Dev. Math.} Springer, New York, 2010, pp.~249--258.

\bibitem{MR2258262}
{\sc Karpenko, N.~A., and Merkurjev, A.~S.}
\newblock Canonical {$p$}-dimension of algebraic groups.
\newblock {\em Adv. Math. 205}, 2 (2006), 410--433.

\bibitem{MR2713320}
{\sc Mathews, B.~G.}
\newblock {\em Canonical dimension of projective homogeneous varieties of inner
  type {A} and type {B}}.
\newblock ProQuest LLC, Ann Arbor, MI, 2009.
\newblock Thesis (Ph.D.)--University of California, Los Angeles.

\bibitem{MR2024643}
{\sc Merkurjev, A.~S.}
\newblock Steenrod operations and degree formulas.
\newblock {\em J. Reine Angew. Math. 565\/} (2003), 13--26.

\bibitem{ASM-ed}
{\sc Merkurjev, A.~S.}
\newblock Essential dimension.
\newblock In {\em Quadratic Forms -- Algebra, Arithmetic, and Geometry},
  vol.~493 of {\em Contemp. Math.} Amer. Math. Soc., Providence, RI, 2009,
  pp.~299--326.

\bibitem{J-invariant}
{\sc Petrov, V., Semenov, N., and Zainoulline, K.}
\newblock {$J$}-invariant of linear algebraic groups.
\newblock {\em Ann. Sci. \'Ecole Norm. Sup. (4) 41\/} (2008), 1023--1053.

\bibitem{MR0224710}
{\sc Tits, J.}
\newblock Classification of algebraic semisimple groups.
\newblock In {\em Algebraic {G}roups and {D}iscontinuous {S}ubgroups ({P}roc.
  {S}ympos. {P}ure {M}ath., {B}oulder, {C}olo., 1965)}. Amer. Math. Soc.,
  Providence, R.I., 1966, pp.~33--62.

\bibitem{Vishik-IMQ}
{\sc Vishik, A.}
\newblock Integral motives of quadrics.
\newblock Preprint of the Max Planck Institute at Bonn, (1998)-13, 82 pages.

\bibitem{MR2377113}
{\sc Vishik, A., and Yagita, N.}
\newblock Algebraic cobordisms of a {P}fister quadric.
\newblock {\em J. Lond. Math. Soc. (2) 76}, 3 (2007), 586--604.

\bibitem{maksim}
{\sc Zhykhovich, M.}
\newblock Motivic decomposability of generalized {S}everi-{B}rauer varieties.
\newblock {\em C. R. Math. Acad. Sci. Paris 348}, 17-18 (2010), 989--992.

\end{thebibliography}

\def\cprime{$'$}

\end{document}